\documentclass{amsart}

\usepackage{amsmath,amsthm,amssymb,amsfonts}
\usepackage{enumerate}
\usepackage[utf8]{inputenc}

\theoremstyle{plain}
\newtheorem{theorem}{Theorem}
\newtheorem{lemma}[theorem]{Lemma}
\newtheorem{claim}[theorem]{Claim}

\theoremstyle{definition}

\newtheorem{remark}[theorem]{Remark}

\newcommand{\R}{\mathbb{R}}
\newcommand{\Q}{\mathbb{Q}}
\newcommand{\N}{\mathbb{N}}

\DeclareMathOperator{\osc}{osc}

\DeclareMathOperator{\diam}{diam}
\providecommand{\restriction}{|}

\begin{document}

\title{A game characterizing Baire class 1 functions}

\author{Viktor Kiss}
\address{Alfr\'ed R\'enyi Institute of Mathematics\\
Hungarian Academy of Sciences\\
P.O. Box 127, H-1364 Budapest, Hungary and Department of Mathematics, Cornell University, Ithaca, NY 14853.}
\email{kiss.viktor@renyi.mta.hu}
\thanks{The author was partially supported by the National Research, Development and Innovation Office	-- NKFIH, grants no.~113047, 104178, 124749, 129211 and 128273, and by the NSF grant DMS-1455272}


\begin{abstract}
  Duparc introduced a two-player game for a function $f$ between zero-dimensional Polish spaces in which Player II has a winning strategy iff $f$ is of Baire class 1. We generalize this result by defining a game for an arbitrary function $f : X \to Y$ between arbitrary Polish spaces such that Player II has a winning strategy in this game iff $f$ is of Baire class 1. Using the strategy of Player II, we reprove a result concerning first return recoverable functions.
\end{abstract}

\maketitle

\section{Introduction}

A \emph{Polish space} is a separable, completely metrizable topological space. A function $f : X \to Y$ between Polish spaces $X$ and $Y$ is called \emph{Baire class 1} if the inverse image $f^{-1}(U)$ of an open subset $U \subseteq Y$ is $F_\sigma$ in $X$, that is, it is the countable union of closed sets. 

The main focus of this paper are games characterizing classes of functions: these are two-player games depending on a function $f$ with the property that one of the players, say Player II has a winning strategy in the game if and only if $f$ belongs to the specific class. In his Ph.D. Thesis \cite{WadgePHD}, Wadge introduced a game that characterizes Lipschitz functions, and another one that characterizes continuous functions, both for self-maps of the Baire space, $\N^\N$. 

The eraser game introduced (essentially) by Duparc \cite{Duparc} characterizes Baire class 1 functions $f : \N^\N \to \N^\N$. Since we wish to generalize this game, we briefly define it here. 

\begin{equation*}
\begin{split}
  &\text{I}  \;\; \quad a_0 \quad \quad a_1 \quad \quad a_2 \quad \quad a_ 3 \quad \quad a_ 4 \; \dots \\
  &\text{II} \quad \quad \quad b_0 \quad \quad b_1 \quad \quad \not{b_1} \quad \quad \not{b_0} \quad \quad b'_0 \; \dots 
\end{split}
\end{equation*}

In the eraser game, at each step, Player I has to play a natural number. Player II either plays a natural number, or erases the last natural number that appears on his board. During the game, Player I builds an element of the Baire space $a \in \N^\N$. Player II is also required to build an element $b = (b_n)_{n \in \N} \in \N^\N$, otherwise he loses. In other words, for every $n \in \N$ there has to be an index $m \in \N$ such that the board of Player II contains at least $n$ natural numbers after the $m$th step of the game, and the first $n$ natural numbers on his board are not erased later. For a fixed function $f: \N^\N \to \N^\N$, Player II wins a run of the eraser game if and only if $f(a) = b$, where $a \in \N^\N$ and $b \in \N^\N$ again denotes the elements of the Baire space built by Player I and Player II, respectively. 

By a result of Duparc, $f : \N^\N \to \N^\N$ is Baire class 1 if and only if Player II has a winning strategy in the eraser game. Carroy \cite{Carroy} proved that if $f$ is not Baire class 1 then Player I has a winning strategy. 

There have been other results for different classes of functions between zero-dimensional Polish spaces. Andretta \cite{Andretta} proved that the so-called back-track game characterizes $\boldsymbol{\Delta}^0_2$-measurable functions, Semmes \cite{Semmes} characterized Borel games, and Nobrega \cite{Baire_xi_games} constructed games characterizing Baire class $\xi$ functions for every countable ordinal $\xi$. For a more thorough introduction on the subject, and results concerning piecewise defined functions, see Motto Ros \cite{Motto Ros}. 

In this paper we construct a two-player game $G_f$ that can be used to characterize Baire class 1 functions between arbitrary Polish spaces, generalizing the eraser game. Let $X$ and $Y$ be Polish spaces, let $d_X$ be a compatible, complete metric on $X$ and let $f : X \to Y$ be an arbitrary function. At the $n$th step of the game, Player I plays $x_n$, then Player II plays $y_n$, 
\begin{equation*}
\begin{split}
  &\text{I}  \;\; \quad x_0 \quad \quad x_1 \quad \quad x_2 \quad \dots \\
  &\text{II} \quad \quad \quad y_0 \quad \quad y_1 \quad \quad y_2 \quad \dots 
\end{split}
\end{equation*}
with the rules that 
\begin{equation}
\label{e:condition of game}
\text{$x_n \in X$, $y_n \in Y$ and $d_X(x_n, x_{n + 1}) \le 2^{-n}$ for every $n \in \N$.}
\end{equation}
From the fact that $d_X$ is complete, it follows that $x_n \to x$ for some $x \in X$. Player II wins a run of the game if and only if $(y_n)$ is convergent and $y_n \to f(x)$. We note that if $X = Y = \N^\N$ and $d_X$ is the metric with $d_X(x, x') = 2^{-n + 1}$, where $n$ is the smallest index with $x(n) \neq x'(n)$, then from a winning strategy of Player II in the eraser game, one can derive a winning strategy for Player II in $G_f$ and vice-versa. 

Our main theorem concerning this game is the following.

\begin{theorem}
  \label{t:main}
  If $f$ is of Baire class 1 then Player II has a winning strategy in $G_f$. If $f$ is not of Baire class 1 then Player I has a winning strategy in $G_f$. In particular, the game $G_f$ is determined. 
\end{theorem}

\begin{remark}
  We note that if we change the rules of the game $G_f$ and leave out the condition that $d_X(x_n, x_{n + 1}) \le 2^{-n}$ and, of course, change the winning condition so that Player II wins if and only if $(x_n)_{n \in \N}$ is divergent or $y_n \to f(x)$ where $x_n \to x$, then Theorem \ref{t:main} does not remain true. To see this, let $f : \R \to \R$ be defined by $f(0) = 1$, $f(x) = 0$ if $x \neq 0$. It is easy to see that $f$ is of Baire class 1, but Player I has a winning strategy in this modified game. 
  
  We sketch the proof of this. The construction of the winning strategy of Player I is similar to the construction in the proof of Theorem \ref{t:main}. Let us fix a sequence $(x^n)_{n \in \N}$ with $x^n \to 0$ and $x^{2n} = 0$, $x^{2n + 1} \neq 0$ for every $n \in \N$. Now let Player I play a sequence $x_0, x_1,  \dots$ with $x_0 = \dots = x_{n_0} = x^0$, $x_{n_0 + 1} = \dots = x_{n_1} = x^1$ etc., where at each step he waits until Player II plays an element $y_{n_k} \in (3/4, 5/4)$ if $k$ is even, and $y_{n_k} \in (-1/4, 1/4)$ if $k$ is odd. One can easily check that the sequence $(x_n)_{n \in \N}$ is either constant $x^k$ after a while with $y_n \not \to f(x^k)$, or $x_n \to 0$ but $|y_{n_k} - y_{n_{k + 1}}| \ge 1/2$, hence the sequence $(y_n)_{n \in \N}$ does not converge. 
\end{remark}

After proving our main result, we investigate the connection of the above game and the notion of first return recoverable functions that was introduced by Darji and Evans \cite{DE}. Let $X$ and $Y$ be Polish spaces, and $d_X$ be a compatible metric on $X$. By a \emph{trajectory}, we mean a dense sequence $(x_n)_{n} \subseteq X$. For any $s \in X$, let $r(B(s, \rho))$ denote the first element of the trajectory $(x_n)_n$ in the open ball $B(s, \rho) = \{s' \in X : d_X(s, s') < \rho\}$. Then the \emph{first return route to $s$}, $(s_n)_n$ is defined as 
\begin{align*}
s_0 &= x_0, \\
s_{n + 1} &= \left\{\begin{array}{lc}
r(B(s, d(s, s_n))) & \text{if } s \neq s_n, \\
s & \text{if } s = s_n.
\end{array} \right.
\end{align*}

A function $f : X \to Y$ is said to be \emph{first return recoverable} if there exists a trajectory $(x_n)_n$ such that for any $s \in X$, we have $f(s_n) \to f(s)$. The definition basically says that the values of $f$ can be recovered using a simple algorithm, knowing the values at only countably many points. 

Darji and Evans showed that every recoverable function is Baire class 1, and if $X$ is a compact metric space then the converse also holds, giving a characterization of Baire class 1 functions. Lecomte \cite{Lecomte} found an example of a Baire 1 function on a Polish ultrametric space which is not recoverable. However, he also showed that if $X$ is an ultrametric space with the property that every strictly decreasing sequence in the range of the ultrametric converges to 0, then every Baire 1 function defined on $X$ is recoverable. Later Duncan and Solecki \cite{DS} gave a characterization of those Polish ultrametric spaces on which every real-valued, Baire class 1 function is recoverable.

From the result of Lecomte it follows easily that if $X$ is a zero-dimensional Polish space then one can find a compatible, complete metric $d$ on $X$ such that $f : X \to Y$ ($Y$ Polish) is first return recoverable with respect to $d$ if and only if $f$ is of Baire class 1. To the best of the author's knowledge, the analogous problem concerning an arbitrary Polish space $X$ is still open. Nonetheless, we give an alternate proof of the result for the zero-dimensional case in Section \ref{s:first return}. We use the strategy of Player II in the above game to extract a suitable trajectory. 

\section{Proof of Theorem \ref{t:main}}

In this section we prove Theorem \ref{t:main}. As we mentioned before, Carroy \cite{Carroy} constructs a winning strategy for Player I in the eraser game (if $f$ is not Baire class 1), and using his ideas one can construct a winning strategy for Player I in $G_f$. We include a construction anyway to keep the paper self-contained. The main difficulty is to construct a winning strategy for Player II when $f$ is of Baire class 1. In the eraser game a strategy is constructed using the fact that a Baire class 1 function $f : \N^\N \to \N^\N$ is always the pointwise limit of a sequence of continuous functions. However, this is not the case for functions between arbitrary Polish spaces, hence we need new ideas to complete the proof. 

\begin{proof}[Proof of Theorem \ref{t:main}]
  Let us fix a compatible, complete metric $d_Y$ for $Y$ and denote the oscillation of $f$ restricted to a closed set $F \subseteq X$ at a point $x \in F$ by 
  \begin{equation}
    \label{e:def of osc}
    \osc_{f \restriction F}(x) = \inf_{U \ni x \text{ open}} \sup\{d_Y(f(x_1), f(x_2)) : x_1, x_2 \in U \cap F\}.
  \end{equation}
  It is easy to check that 
  \begin{equation}
    \label{e:osc = 0 <=> continuous} \osc_{f \restriction F}(x) = 0 \Leftrightarrow \text{$f\restriction F$ is continuous at $x$},
  \end{equation}
  and that $\osc_{f \restriction F}$ is upper semi-continuous, hence for every $\varepsilon > 0$, 
  \begin{equation}
    \label{e:set with large osc is closed}
    \{x \in F : \osc_{f \restriction F}(x) \ge \varepsilon\} \text{ is closed}. 
  \end{equation}
  
  To prove the first assertion of the theorem, suppose that $f$ is of Baire class 1. We need to define a winning strategy for Player II, hence a method of coming up with $y_n$ if $x_0, x_1, \dots, x_n$ are already given. Of course, the possible limit point $x$ of the sequence $(x_n)_{n \in \N}$ is in the closed ball $\overline{B}(x_n, 2^{1 - n}) = \{x' \in X : d_X(x', x_n) \le 2^{1 - n}\}$. The idea of the proof is to pick $y_n$ as the image of a point in $\overline{B}(x_n, 2^{1 - n})$ at which $f$ behaves ``badly''. We note here that for some functions, including the modified Dirichlet function (that is the function $g : \R \to \R$ with $g(p/q) = 1/q$ if $p$ and $q$ are relatively prime and $q > 0$, and $g(x) = 0$ if $x \not \in \Q$ or $x = 0$) it would be sufficient to pick $y_n$ as the image of a point in $\overline{B}(x_n, 2^{1 - n})$ with the largest oscillation (or a sufficiently large oscillation), because the function restricted to the set of points with large oscillation is continuous. However, in the general case the restriction may not be continuous, and we need to do an iterative construction.

  With the help of the set
  $$R = \{0\} \cup \left\{1/n : n \in \N, n > 0\right\}$$
  and the function 
  \begin{equation}
    \label{e:o def}
    o_f(F) = \left\{\begin{matrix} \max\{r \in R : \exists \, x \in F \, (\osc_{f \restriction F}(x) \ge r)\} && \text{if $F \neq \emptyset$}, \\ 
    0 && \text{if $F = \emptyset$},\end{matrix}\right.
  \end{equation}
  we define a derivative operation on the family of closed subsets of $X$ by
  \begin{equation*}
    \label{e:D derivative def}
    D(F) = \left\{\begin{matrix} \{x \in F : \osc_{f \restriction F}(x) \ge o_f(F)\} && \text{if $o_f(F) > 0$}, \\ \emptyset && \text{if $o_f(F) = 0$}. \end{matrix} \right.
  \end{equation*}
  Using \eqref{e:set with large osc is closed}, $D(F)$ is closed for every closed set $F \subseteq X$. Using Baire's theorem that a Baire class 1 function has a point of continuity restricted to every non-empty closed subset (see e.g. \cite[Theorem 24.15]{K}), if $o_f(F) > 0$, or equivalently by \eqref{e:osc = 0 <=> continuous}, if $f\restriction F$ is not continuous then $D(F) \subsetneq F$. If $o_f(F) = 0$ but $F \neq \emptyset$ then we also have $D(F) = \emptyset \subsetneq F$, hence 
  \begin{equation}
    \label{e:D(F) neq F} 
    F \neq \emptyset \Rightarrow D(F) \subsetneq F.
  \end{equation}
  We also note here that 
  \begin{equation}
  \label{e:F < F' -> o_f(F) <= o_f(F')}
  F \subseteq F' \Rightarrow o_f(F) \le o_f(F'), \text{ and}
  \end{equation}
  \begin{equation}
    \label{e:F not empty D(F) empty -> o_f(F) = 0}
    F \neq \emptyset \land D(F) = \emptyset \Rightarrow o_f(F) = 0. 
  \end{equation}
  
  Now we define the iterated derivative of a closed subset $F \subseteq X$ the usual way for each $\alpha < \omega_1$, that is,
  \begin{equation*}
  \begin{split}
    D^0(F) &= F, \\
    D^{\alpha + 1}(F) &= D(D^{\alpha}(F)), \\
    D^{\alpha}(F) &= \bigcap_{\beta < \alpha} D^{\beta}(F) 
      \text{ if $\alpha$ is limit.}  
  \end{split}
  \end{equation*}
  It can be easily shown by transfinite induction on $\beta$ that 
  \begin{equation}
  \label{e:al < be -> D^al > D^be}
  \text{$\alpha < \beta \Rightarrow D^\alpha(F) \supseteq D^\beta(F)$ for every closet set $F \subseteq X$.}
  \end{equation}
   
  Using \eqref{e:D(F) neq F} and the fact that strictly decreasing transfinite sequences of closed subsets of a Polish space are always countable (see e.g. \cite[Theorem 6.9]{K}), for every closed set $F \subseteq X$ there exists a countable ordinal $\lambda$ with $D^\lambda(F) = \emptyset$. Let us denote the smallest such $\lambda$ by $\lambda(F)$. 

  The derivative operation will be used to construct the closed sets in the following, main lemma, where we say that a sequence of closed sets $(F_n)_{n \in \N}$ \emph{converges} to a point $x$, if any neighborhood of $x$ contains all, but finitely many of the $F_n$'s.
  \begin{lemma}
    \label{l:o_f(F_n) -> 0, F_n -> x}
    Let $(F_n)_{n \in \N}$ be a decreasing sequence of nonempty closed sets converging to $x$ with $o_f(F_n) \to 0$. Suppose that the sequence $(y_n)_{n \in \N} \subseteq Y$ satisfies for each $n \ge 1$ that either $y_n \in f(F_n)$ or $y_n = y_{n - 1}$, with the first possibility occurring infinitely many times. Then $y_n \to f(x)$. 
  \end{lemma}
  \begin{proof}
    It is clear from the facts that $(F_n)_{n \in \N}$ is decreasing sequence of closed sets with $F_n \to x$ that $\bigcap_n F_n = \{x\}$. Now let $\varepsilon > 0$, $\varepsilon < 1$ be fixed, we need to find $n_0 \in \N$ with $d_Y(y_n, f(x)) \le \varepsilon$ for every $n \ge n_0$. Let $n_1 \in \N$ be large enough so that $o_f(F_n) < \varepsilon / 2$ for every $n \ge n_1$. Since $x \in F_{n_1}$, it easily follows from the definition of $o_f$, \eqref{e:o def}, that $\osc_{f\restriction F_{n_1}}(x) < \varepsilon$, hence for small enough $\delta > 0$, 
    \begin{equation}
      \label{e:d(x', x) < delta, x' in F_n_1 -> d(f(x'), f(x)) < eps}
      \text{$d_X(x', x) \le \delta$ and $x' \in F_{n_1}$ imply $d_Y(f(x'), f(x)) < \varepsilon$. }
    \end{equation}

    Now let $n_0 \ge n_1$ be large enough so that 
    \begin{equation}
      \label{e:diam(F_n) <= delta}
      \text{$\diam(F_n) \le \delta$ for every $n \ge n_0$,}
    \end{equation}
    and since $y_n \in f(F_n)$ for infinitely many $n$, we can also suppose that 
    \begin{equation}
      \label{e:y_n_0 in F_n_0}
      y_{n_0} \in f(F_{n_0}).
    \end{equation}
    Let $n \ge n_0$ be fixed, we need to show that $d_Y(y_n, f(x)) < \varepsilon$. If $y_n \in f(F_n)$ then $d_Y(y_n, f(x)) < \varepsilon$ using \eqref{e:diam(F_n) <= delta}, \eqref{e:d(x', x) < delta, x' in F_n_1 -> d(f(x'), f(x)) < eps} and the fact that $n \ge n_1$ implies $F_n \subseteq F_{n_1}$. If $y_n \not \in f(F_n)$ then using \eqref{e:y_n_0 in F_n_0}, there exists $k < n$ with $k \ge n_0$ such that $y_n = y_k \in f(F_k)$. Then, as we already saw, $d_Y(y_k, f(x)) < \varepsilon$, thus the proof of the lemma is complete. 
  \end{proof}  
  
  During the construction of a winning strategy for Player II, we use the notation 
  \begin{equation}
  \begin{split}
    \label{e:def of X, r, lambda}
    X_i^\alpha &= D^\alpha(\overline{B}(x_i, 2^{1 - i})), \\
    \lambda_i &= \lambda(X_i^0) = \lambda(\overline{B}(x_i, 2^{1 - i})),
  \end{split}
  \end{equation}
  where, as before, $\overline{B}(x_i, 2^{1 - i})$ denotes the closed ball $\{x' \in X : d_X(x', x_i) \le 2^{1 - i}\}$. Note that using the rules of the game \eqref{e:condition of game} and that $x_i \to x$, we have 
  \begin{equation}
    \label{e:x in B(x_n,...)}
    \text{$x \in \overline{B}(x_i, 2^{1 - i})$ for every $i$}. 
  \end{equation}
  For $i \ge 1$ let $\gamma_i$ denote the smallest ordinal $\gamma < \omega_1$ such that $o_f(X_i^\gamma) \neq o_f(X_{i - 1}^\gamma)$ if such an ordinal exists, and let $\gamma_i = \omega_1$ otherwise. 
  
  Before defining the strategy for Player II, we collect a couple of simple properties of the sets $X_n^\alpha$ that we will use in our proof. Let $\diam(H)$ denote the diameter of the set $H \subseteq X$, that is, $\diam(H) = \sup\{d_X(x', x'') : x', x'' \in H\}$.
  \begin{claim}
  \label{c:basic properties}
  \begin{enumerate}[(i)]
    \item\label{p:al < be -> X_n^al > X_n^be} $\forall n \left(\alpha < \beta \Rightarrow X_n^\alpha \supseteq X_n^\beta\right)$,
    \item\label{p:diameter} $\forall n\; \forall \alpha \left(\diam(X_n^\alpha) \le 2^{2 - n}\right)$,
    \item\label{p:X_n^al < X_n-1^al} $\forall n \ge 1\; \forall \alpha \le \gamma_n \left(\alpha < \omega_1 \Rightarrow X_n^\alpha \subseteq X_{n - 1}^\alpha\right),$
    \item\label{p:o_f decreases at gamma_n} $\forall n \ge 1 \left(\gamma_n < \omega_1 \Rightarrow o_f(X_n^{\gamma_n}) < o_f(X_{n - 1}^{\gamma_n})\right).$
  \end{enumerate}
  \end{claim}
  \begin{proof}
    \eqref{p:al < be -> X_n^al > X_n^be} is the application of \eqref{e:al < be -> D^al > D^be} with $F = X_n^0$. 
    
    To see \eqref{p:diameter}, note that $X_n^\alpha \subseteq X_n^0$ for every $\alpha < \omega_1$ using \eqref{p:al < be -> X_n^al > X_n^be}, hence $\diam(X_n^\alpha) \le \diam(X_n^0) = \diam(\overline{B}(x_n, 2^{1 - n})) \le 2^{2 - n}$. 
    
    We prove \eqref{p:X_n^al < X_n-1^al} by transfinite induction on $\alpha$. It holds for $\alpha = 0$, since $X_{n - 1}^0 = \overline{B}(x_{n - 1}, 2^{2 - n}) \supseteq \overline{B}(x_n, 2^{1 - n}) = X_n^0$, using that $d_X(x_{n - 1}, x_n) \le 2^{1 - n}$ by \eqref{e:condition of game}. It is clear for a limit $\alpha$ that if $X_n^\beta \subseteq X_{n - 1}^\beta$ for every $\beta < \alpha$ then $X_n^\alpha \subseteq X_{n - 1}^\alpha$. It remains to show that if $X_n^\alpha \subseteq X_{n - 1}^\alpha$ and $\alpha + 1 \le \gamma_n$ then $X_n^{\alpha + 1} \subseteq X_{n - 1}^{\alpha + 1}$. From $\alpha + 1 \le \gamma_n$ it follows that $o_f(X_n^\alpha) = o_f(X_{n - 1}^\alpha)$. If $o_f(X_n^\alpha) = o_f(X_{n - 1}^\alpha) = 0$ then $X_{n}^{\alpha + 1} = X_{n - 1}^{\alpha + 1} = \emptyset$. 
    Otherwise, $X_n^{\alpha + 1} = 
    \{x \in X_n^\alpha : \osc_{f \restriction X_n^\alpha}(x) \ge o_f(X_n^\alpha)\} \subseteq 
    \{x \in X_{n}^\alpha : \osc_{f \restriction X_{n - 1}^\alpha}(x) \ge o_f(X_n^\alpha)\} \subseteq 
    \{x \in X_{n - 1}^\alpha : \osc_{f \restriction X_{n - 1}^\alpha}(x) \ge o_f(X_n^\alpha)\} = 
    X_{n - 1}^{\alpha + 1}$ using twice the inductive assumption $X_n^\alpha \subseteq X_{n - 1}^\alpha$ and also the fact that $F \subseteq F'$ implies $\osc_{f\restriction F}(x') \le \osc_{f \restriction F'}(x')$ for every $x' \in F$. 
    
    To see \eqref{p:o_f decreases at gamma_n}, note that $X_n^{\gamma_n} \subseteq X_{n - 1}^{\gamma_n}$ by \eqref{p:X_n^al < X_n-1^al}, hence $o_f(X_n^{\gamma_n}) \le o_f(X_{n - 1}^{\gamma_n})$ by \eqref{e:F < F' -> o_f(F) <= o_f(F')}, and by the definition of $\gamma_n$ we have $o_f(X_n^{\gamma_n}) \neq o_f(X_{n - 1}^{\gamma_n})$ yielding \eqref{p:o_f decreases at gamma_n}. 
  \end{proof}  
  
  Now we define the strategy for Player II, that is, we define $y_n \in Y$, given $x_0, \dots, x_n$ and $y_0, \dots, y_{n - 1}$. 
  
  \noindent{\bf Case (a):} \emph{$\lambda_n = \alpha + 1$ is successor.} In this case let $y_n \in f(X_n^{\alpha})$ be arbitrary. 
  
  \noindent{\bf Case (b):} \emph{$\lambda_n$ is limit and $\gamma_n < \lambda_n$.} In this case let $y_n \in f(X_n^{\gamma_n})$ be arbitrary. 
  
  \noindent{\bf Case (c):} \emph{$\lambda_n$ is limit and $\gamma_n \ge \lambda_n$.} Then let $y_n \in Y$ be arbitrary if $n = 0$, and let $y_n = y_{n - 1}$ otherwise. 
  
  This concludes the definition of the strategy for Player II.
  
  \begin{remark}
    \label{r:strategy for ultrametric}
    We insert here a remark to help us in the proof of Theorem \ref{t:first return}. Suppose $(X, d_X)$ happens to be an ultrametric space. It is well-known that in this case closed balls of radius $r$ form a partition of $X$ for any $r > 0$. Hence for each $n$, there are only countably many closed balls of the form $\overline{B}(x_n, 2^{1 - n})$, and each sits inside a unique closed ball of the form $\overline{B}(x_{n - 1}, 2^{2 - n})$. Hence, given $n$ and a closed ball of the form $\overline{B}(x_n, 2^{1 - n})$, we can calculate $\lambda_n$ and $\gamma_n$ without knowing $x_n$ or $x_{n - 1}$. Thus, we can also calculate the move $y_n$ of Player II. In Case (a) and Case (b) we can even choose for every $n$ and every ball $B = \overline{B}(x_n, 2^{1 - n})$ some element $x_{B, n} \in B$ such that Player II respects the strategy by playing $y_n = f(x_{B, n})$.
  \end{remark}
  
  We divide the proof of the correctness of the strategy into multiple cases. 
  
  \noindent{\bf Case (1):} \emph{for infinitely many $n$, $\gamma_n < \omega_1$}. Let 
  \begin{equation}
    \label{e:gamma def}
    \gamma = \min \{ \eta : \{n \in \N : \gamma_n \le \eta\} \text{ is infinite}\}. 
  \end{equation}
  Since we are in Case (1), $\gamma < \omega_1$. 
  
  \noindent{\bf Case (1a):} \emph{$\gamma_n \ge \gamma$ for all, but finitely many $n$}. Considering the assumptions of Case (1) and Case (1a), it is easy to see that there exists $m \in \N$ so that 
  \begin{equation}
    \label{e:m choice}
    \text{$\gamma_m = \gamma$, $\gamma_n \ge \gamma$ for every $n \ge m$ and $\gamma_n = \gamma$ for infinitely many $n$.}
  \end{equation}
  Now we use Lemma \ref{l:o_f(F_n) -> 0, F_n -> x}, with $X_{m + n}^\gamma$ in place of $F_n$ and $y_{m + n}$ in place of $y_n$. We need to check that the conditions of the lemma hold to complete Case (1a). It is easy to see that $X^\gamma_{m + n} \to x$ as $n \to \infty$, using \eqref{e:def of X, r, lambda} and \eqref{e:x in B(x_n,...)}. One can show by induction using \eqref{p:X_n^al < X_n-1^al} of Claim \ref{c:basic properties} that $X_j^\gamma \subseteq X_i^\gamma$ for every $i, j \ge m$ with $i \le j$, hence $(X_{m + n}^\gamma)_{n \in \N}$ is decreasing. 
  Using this observation and \eqref{e:F < F' -> o_f(F) <= o_f(F')} it follows that $o_f(X_j^\gamma) \le o_f(X_i^\gamma)$ for every $i, j \ge m$ with $i \le j$. Since $\gamma_n = \gamma$ for infinitely many $n$ by \eqref{e:m choice},
  $o_f(X_{m + n + 1}^\gamma) < o_f(X_{m + n}^\gamma)$ for infinitely many $n$, hence
  \begin{equation}
    \label{e:o_f > 0, o_f -> 0}
    \text{$o_f(X_n^\gamma) > 0$ for every $n \ge m$ and $o_f(X_n^\gamma) \to 0$.}
  \end{equation}

  It also follows that $X_{m + n}^\gamma \neq \emptyset$. Thus the conditions of Lemma \ref{l:o_f(F_n) -> 0, F_n -> x} concerning only $(X_{m + n}^\gamma)_{n \in \N}$ hold. 
  
  It remains to check that the conditions concerning $(y_{m + n})_{n \in \N}$ also hold. Using \eqref{e:o_f > 0, o_f -> 0}, $\gamma < \lambda_n$ for all $n \ge m$, hence $\gamma_n < \lambda_n$ for infinitely many $n$ by \eqref{e:m choice}. It follows that $y_{m + n}$ is chosen according to Case (a) or Case (b) for infinitely many $n$. Since $\gamma < \lambda_n$ and $\gamma \le \gamma_n$ for all $n \ge m$, $y_{m + n} \in f(X_{m + n}^\gamma) = f(F_n)$ in both cases. If $y_{m + n}$ is chosen according to Case (c) then $y_{m + n} = y_{m + n - 1}$, hence the conditions of Lemma \ref{l:o_f(F_n) -> 0, F_n -> x} hold, the strategy is winning for Player II in Case (1a).

  \noindent{\bf Case (1b):} \emph{$\gamma_n < \gamma$ for infinitely many $n$.} 
  \begin{claim}
    \label{c:prop of n_k}
    There exists a strictly increasing sequence $(n_k)_{k \in \N}$ such that $\gamma_{n_k}\to \gamma$ and $\gamma_n > \gamma_{n_k}$ for every $n > n_k$. In particular, the sequence $(\gamma_{n_k})_{k \in \N}$ is also strictly increasing. 
  \end{claim}
  \begin{proof}
    Let $n$ be arbitrary with $\gamma_n < \gamma$, and let $n_0 = \max\{m \ge n : \gamma_m \le \gamma_n\}$. The maximum exists using the fact that $\gamma_n < \gamma$ and the definition of $\gamma$ \eqref{e:gamma def}. Now let $n > n_0$ be arbitrary with $\gamma_n < \gamma$ and let $n_1 = \max\{m \ge n : \gamma_m \le \gamma_n\}$. Iterating this construction we get a strictly increasing sequence $(n_k)_{k \in \N}$ with the properties that $\gamma_{n_k} < \gamma$ and $\gamma_n > \gamma_{n_k}$ if $n > n_k$. Then $\sup\{\gamma_{n_k} : k \in \N\} \le \gamma$, and using \eqref{e:gamma def} again, one can easily see that $\gamma_{n_k} \to \gamma$.
  \end{proof}
  
  Now we fix such a sequence $(n_k)_{k \in \N}$ and use Lemma \ref{l:o_f(F_n) -> 0, F_n -> x} again, with $y_{n_0 + n}$ in place of $y_n$, and for $n \in \N$ taking the unique $k \in \N$ with $n_k \le n_0 + n < n_{k + 1}$ we use $X_{n_0 + n}^{\gamma_{n_k}}$ in place of $F_{n}$. It is easy to check that we defined the set $F_n$ for every $n \in \N$. We now check that the conditions of Lemma \ref{l:o_f(F_n) -> 0, F_n -> x} hold. It is clear that $F_n \to x$. To prove that $(F_{n})_{n \in \N}$ is a decreasing sequence, let $n \in \N$, we need to show that $F_{n + 1} \subseteq F_{n}$. Let $k \in \N$ be the unique natural number with $n_k \le n_0 + n < n_{k + 1}$. If $n_0 + n + 1 < n_{k + 1}$ then $F_{n + 1} = X_{n_0 + n + 1}^{\gamma_{n_k}}$, showing that $F_{n + 1} = X_{n_0 + n + 1}^{\gamma_{n_k}}\subseteq X_{n_0 + n}^{\gamma_{n_k}} = F_n$ using that $\gamma_{n_0 + n + 1} > \gamma_{n_k}$ provided by Claim \ref{c:prop of n_k}, and \eqref{p:X_n^al < X_n-1^al} of Claim \ref{c:basic properties}. If $n_0 + n + 1 = n_{k + 1}$ then $F_{n + 1} = X_{n_{k + 1}}^{\gamma_{n_{k + 1}}} \subseteq X_{n_{k + 1} - 1}^{\gamma_{n_{k + 1}}} = X_{n_0 + n}^{\gamma_{n_{k + 1}}} \subseteq X_{n_0 + n}^{\gamma_{n_k}} = F_n$ using \eqref{p:X_n^al < X_n-1^al} and \eqref{p:al < be -> X_n^al > X_n^be} of Claim \ref{c:basic properties}. 
  
  Now we show that $o_f(F_n) \to 0$. Since $(F_n)_{n \in \N}$ is decreasing, it is enough to show that $o_f(F_{n_k - n_0}) \to 0$ by \eqref{e:F < F' -> o_f(F) <= o_f(F')}. For every $k \in \N$ one can easily show by induction that $o_f(F_{n_k - n_0}) = o_f(X_{n_k}^{\gamma_{n_k}}) =  o_f(X_{n_{k + 1} - 1}^{\gamma_{n_{k}}})$, using the properties of $(n_k)_{k \in \N}$ provided by Claim \ref{c:prop of n_k} and \eqref{p:X_n^al < X_n-1^al} of Claim \ref{c:basic properties}. Thus, $o_f(F_{n_k - n_0}) = o_f(X_{n_k}^{\gamma_{n_k}}) = o_f(X_{n_{k + 1} - 1}^{\gamma_{n_{k}}}) \ge o_f(X_{n_{k + 1} - 1}^{\gamma_{n_{k + 1}}}) > o_f(X_{n_{k + 1}}^{\gamma_{n_{k + 1}}}) = o_f(F_{n_{k + 1} - n_0})$, using \eqref{p:al < be -> X_n^al > X_n^be} and \eqref{p:X_n^al < X_n-1^al} of Claim \ref{c:basic properties}. Hence, using also that the range of $o_f$ is $R$ by \eqref{e:o def}, $o_f(F_{n_k - n_0}) \to 0$ as $k \to \infty$, thus $o_f(F_n) \to 0$. Moreover, we also see that 
  \begin{equation}
    \label{e:o_f(F_n) > 0 1b}
    \text{$o_f(F_n) > 0$ for every $n$,}
  \end{equation}
  using again that $(F_n)_{n \in \N}$ is decreasing and \eqref{e:F < F' -> o_f(F) <= o_f(F')}, hence clearly $F_n \neq \emptyset$ for every $n \in \N$. 
   
  It remains to show that $y_{n_0 + n}$ satisfies the conditions of Lemma \ref{l:o_f(F_n) -> 0, F_n -> x}. Let $n$ be fixed and let $k$ be the unique natural number with $n_k \le n_0 + n < n_{k + 1}$. Since $\gamma_{n_0 + n} \ge \gamma_{n_k}$, if $y_{n_0 + n}$ was chosen according to Case (a) or Case (b) then $y_{n_0 + n} \in f(X_{n_0 + n}^{\gamma_{n_k}}) = f(F_n)$, using again that $X_{n_0 + n}^{\gamma_{n_k}} = F_n \neq \emptyset$. If $y_{n_0 + n}$ was chosen according to Case (c) and $n \ge 1$ then $y_{n_0 + n} = y_{n_0 + n - 1}$. Thus the conditions of Lemma \ref{l:o_f(F_n) -> 0, F_n -> x} are satisfied, hence $y_n \to f(x)$. 
  
  Note that Case (1a) and Case (1b) covers all subcases of Case (1), hence it remains to show that $y_n \to f(x)$ even in the following case.
  
  \noindent{\bf Case (2):} \emph{$\gamma_n = \omega_1$ for all, but finitely many $n$.} Let $m \in \N$ be large enough so that 
  \begin{equation}
    \label{e:(2) gamma_n = om_1 if n >= m}
    \text{$\gamma_n = \omega_1$ for every $n \ge m$.}
  \end{equation}
  From this fact using also \eqref{p:al < be -> X_n^al > X_n^be} and \eqref{p:X_n^al < X_n-1^al} of Claim \ref{c:basic properties} one can easily show first by induction on $j$ and then on $\beta$ that
  \begin{equation}
    \label{e:(2) X_i^al > X_j^be}
    m \le i \le j, \alpha \le \beta \Rightarrow X_i^\alpha \supseteq X_j^\beta.
  \end{equation}
  It follows easily that if $m \le i \le j$ then $\lambda_i \ge \lambda_j$, hence, using that the ordinal numbers are well-ordered there exists $\lambda < \omega_1$ and $M \ge m$ such that 
  \begin{equation}
    \label{e:(2) lam_n = lam if n >= M}
    n \ge M \Rightarrow \lambda_n = \lambda. 
  \end{equation}
  \begin{claim}
    \label{c:lambda is successor}
    $\lambda$ is successor.   
  \end{claim}
  \begin{proof}
    We first show that $x \in X_n^\alpha$ for every $n \ge M$ and $\alpha < \lambda$. Let $\alpha < \lambda$ be fixed. Using \eqref{e:(2) X_i^al > X_j^be} and the fact that $\diam(X_n^\alpha) \le \diam(X_n^0) = \diam(\overline{B}(x_n, 2^{1 - n})) \le 2^{2 - n}$, $(X_n^\alpha)_{n \ge M}$ is a decreasing sequence of closed sets with $\diam(X_n^\alpha) \to 0$. They are nonempty using \eqref{e:(2) lam_n = lam if n >= M} and that $\alpha < \lambda$, hence, there is a unique $x^\alpha \in X$ with $\{x^\alpha\} = \bigcap_{n \ge M} X_n^\alpha$. Using again the fact that $X_n^\alpha \subseteq X_n^0 = \overline{B}(x_n, 2^{1 - n})$, $d_X(x^\alpha, x) \le 2^{2 - n}$ for every $n \in \N$, hence $x = x^\alpha$.
    
    Now the Claim follows, as for any $n \ge M$, if $\lambda$ is limit then $\emptyset = X_n^{\lambda_n} = X_n^\lambda = \bigcap_{\alpha < \lambda} X_n^\alpha \supseteq \{x\} \neq \emptyset$, a contradiction.
  \end{proof}
  
  Let $\lambda = \alpha + 1$, hence 
  \begin{equation}
    \label{e:(2) lam = lam_n = alpha + 1}
    \lambda_n = \lambda = \alpha + 1 \text{ for every $n \ge M$}.
  \end{equation}
  We use Lemma \ref{l:o_f(F_n) -> 0, F_n -> x} again to prove that $y_n \to f(x)$ with $X_{M + n}^\alpha$ in place of $F_n$ and $y_{M + n}$ in place of $y_n$. We first check that the conditions of the lemma hold. \eqref{e:(2) X_i^al > X_j^be} and the fact that $M \ge m$ shows that $(F_n)_{n \in \N} = (X_{M + n}^\alpha)_{n \in \N}$ is a decreasing sequence of closed sets. \eqref{e:(2) lam = lam_n = alpha + 1} implies that each $F_n$ is non-empty. Since $D(X_n^\alpha) = X_n^\lambda = \emptyset$, $o_f(X_n^\alpha) = 0$ by \eqref{e:F not empty D(F) empty -> o_f(F) = 0}. The fact that $F_n \to x$ follows easily from from the construction. 
  
  It is clear that each $y_{M + n}$ was chosen according to Case (a), hence $y_{M + n} \in f(X_{M + n}^\alpha) = f(F_n)$, showing that the conditions of Lemma \ref{l:o_f(F_n) -> 0, F_n -> x} are satisfied. The conclusion of the lemma ensures that $y_n \to f(x)$, completing the analysis of Case (2). Thus, the proof of the first assertion of the theorem is complete. 
  
  \begin{remark}
    \label{r:infinitely many times Case (a) or (b)}
    Note that in both cases, infinitely many of the points $y_n$ was chosen according to Case (a) or Case (b). 
  \end{remark}
  
  It remains to show that if $f$ is not of Baire class 1, then Player I has a winning strategy. We use Baire's theorem again, that states that a function is of Baire class 1 if and only if it has a point of continuity restricted to every non-empty closed set (see e.g. \cite[Theorem 24.15]{K}). Hence, there is a non-empty closed set $F \subseteq X$ such that $f\restriction F$ is not continuous at any point of $F$. Then $\osc_{f \restriction F} (x) > 0$ for every $x \in F$ by \eqref{e:osc = 0 <=> continuous}, hence $F = \bigcup_n F_n$, where $$F_n = \bigcup_n \left\{x \in F : \osc_{f \restriction F}(x) \ge \frac{1}{n}\right\}.$$ Using \eqref{e:set with large osc is closed}, $F_n$ is closed for every $n$. Baire's category theorem implies that there exists $n \in \N$ such that $F_n$ is dense in an open portion of $F$. Let us fix such an $n$, then using that $F_n$ is closed, there exists $U \subseteq X$ open with $\emptyset \neq U \cap F = U \cap F_n \subseteq F_n$. Let $C$ be the closure of $U \cap F_n$ and let $\varepsilon = \frac{1}{n}$. Then $C \subseteq F_n$. We first show that 
  \begin{equation}
    \label{e:osc C >= eps}
    \text{$\osc_{f \restriction C}(x) \ge \varepsilon$ for every $x \in C$.}
  \end{equation}
  Indeed, if $x \in U \cap F_n$ then one can easily see that the oscillation of $x$ is independent of the values of $f$ outside $U$, hence $\osc_{f\restriction C}(x) = \osc_{f \restriction U \cap F_n}(x) = \osc_{f \restriction U \cap F}(x) = \osc_{f \restriction F}(x) \ge \varepsilon$ using that $U \cap F = U \cap F_n$ and that $\varepsilon = \frac{1}{n}$. Now using that $\{x \in C : \osc_{f \restriction C} \ge \varepsilon\}$ is closed by \eqref{e:set with large osc is closed}, it necessarily contains $C$, showing \eqref{e:osc C >= eps}. 
  
  Now we construct a strategy for Player I. Let Player I play an arbitrary element $x_0 = x^0 \in C$. Then Player I plays $x_0 = x_1 = \dots = x^0$ until Player II first plays an elements $y_{n} \in  B(f(x^0), \varepsilon / 7)$, where $B(f(x^0), \varepsilon / 7)$ denotes the open ball $\{y \in Y : d_Y(y, f(x^0)) < \varepsilon / 7)\}$. So let $n_0$ be the smallest natural number with $y_{n_0} \in B(f(x^0), \varepsilon / 7)$, if such a number exists. If no such number exists then Player I plays $x_n = x^0$ at every step of the game and Player II plays a sequence $(y_n)_{n \in \N}$ with $y_n \not \to f(x)$. So we can suppose that at some point, Player II plays $y_{n_0} \in B(f(x^0), \varepsilon / 4)$. Then Player I responds with $x_{n_0 + 1} = x^1 \in B(x^0, 2^{-n_0}) \cap C$ and $d_Y(f(x^1), f(x^0)) \ge \varepsilon\cdot\frac{3}{7}$. Such an element $x^1$ exists using \eqref{e:osc C >= eps} and the definition of the oscillation, \eqref{e:def of osc}. 
  
  Now let Player I play $x^1$ until Player II plays elements outside of the ball $B(f(x^1), \varepsilon / 7)$. So let $n_1$ be the smallest natural number with $y_{n_1} \in B(f(x^1), \varepsilon / 7)$, if such a number exists. If no such number exists then Player I plays $x^1$ indefinitely, with Player II playing a sequence $y_n$ with $y_n \not \to f(x^1)$. Hence, we can suppose that such an index $n_1$ exists. Then we note that from $d_Y(f(x^1), f(x^0)) \ge \varepsilon\cdot\frac{3}{7}$, $y_{n_0} \in B(f(x^0), \varepsilon / 7)$ and $y_{n_1} \in B(f(x^1), \varepsilon / 7)$ it follows that $d_Y(y_{n_0}, y_{n_1}) \ge \varepsilon / 7$. Now we pick $x^2 = x_{n_1 + 1}$ with $x^2 \in B(x^1, 2^{-n_1}) \cap C$ and $d_Y(f(x^2), f(x^1)) \ge \varepsilon\cdot\frac{3}{7}$. Again, the existence of such $x^2$ is ensured by \eqref{e:osc C >= eps}. Iterating the construction, either at some point $k$, when Player I plays $x^k = x_{n_{k - 1} + 1} = x_{n_{k - 1} + 2} \dots$ Player II plays elements $y_{n_{k - 1} + 1}, y_{n_{k - 1} + 2}, \dots  \not \in B(f(x^k), \varepsilon / 7)$ and loses, or an infinite sequence $(n_k)_{k \in \N}$ is constructed with $d_Y(y_{n_k}, y_{n_{k + 1}}) \ge \varepsilon / 7$ for every $k \in \N$, meaning that Player II loses in this case also, finishing the proof of the second assertion of the theorem. Thus, the proof of the theorem is complete. 
\end{proof}

\section{First return recoverable functions}
\label{s:first return}

In this section we reprove the following theorem from the theory of first return recoverable functions, that also follows from a result of Lecomte \cite[Theorem 8]{Lecomte}. 

\begin{theorem}
  \label{t:first return}
  Let $X$ be a zero-dimensional Polish space. Then there is a compatible, complete metric $d_X$ on $X$ with the following property: Every Baire class 1 function $f : X \to Y$, $Y$ a Polish space, is first return recoverable.
\end{theorem}
\begin{proof}
  Let $d_X$ be a compatible, complete metric on $X$ such that the range of $d_X$ is contained in $\{2^{1 - n} : n \in \N\} \cup \{0\}$. Such a metric can be obtained using the fact that $X$ is homeomorphic to a closed subspace of the Baire space (see e.g.~\cite[Theorem 7.8]{K}). 
  
  To show that $d_X$ works, let $f : X \to Y$ be a Baire class 1 function, where $Y$ is an arbitrary Polish space. Player II has a winning strategy in the game $G_f$ using Theorem \ref{t:main}. Using the strategy, the fact that $(X, d_X)$ is an ultrametric space and Remark \ref{r:strategy for ultrametric}, we pick a point $x_{B, n} \in B$ for every $n \in \N$ and closed ball $B$ of radius $2^{1 - n}$ in two steps. Let $n \in \N$ and a closed ball $B = \overline{B}(x_n, 2^{1 - n})$ of radius $2^{1 - n}$ be given, and suppose that the strategy of Player II, as given in the proof of Theorem \ref{t:main}, defines its move for the $n$th step, after I played $x_n$, according to Case (a) or Case (b). Note that by Remark \ref{r:strategy for ultrametric} this depends only on $n$ and $B$, and in these two cases, we can pick $x_{B, n} \in B$ such that playing $f(x_{B, n})$ respects the winning strategy of II. If $n = 0$ and $B = X$, we pick an arbitrary point $x_{X, 0}$ if we have not done so already. The points picked so far will be referred to as \emph{originally picked}. 
  
  If for some $n$ and $B$ we have not picked a point $x_{B, n}$ already, we do so in the following way. Let $d_Y$ be a compatible metric for $Y$. Let $m < n$ be the largest so that if $B'$ is unique closed ball of radius $2^{1 - m}$ containing $B$ then a point $x_{B', m}$ was picked originally. Now, if $x_{B', m} \in B$ then we let $x_{B, n} = x_{B', m}$, otherwise we choose an arbitrary point $x_{B, n} \in B$ such that $d_Y(f(x_{B', m}), f(x_{B, n})) < \frac{1}{n} + \inf\{d_Y(f(x_{B', m}), f(x')) : x' \in B\}$. 
  
  As pointed out in Remark \ref{r:strategy for ultrametric}, this way we picked only countably many points. Let us order these in a sequence $(x_n)_n$ in a way that if $x_i$ corresponds to $x_{B, n}$ and $x_j$ corresponds to $x_{B', n'}$ with $B \subseteq B'$ and $n > n'$ then $i > j$. 
  
  It is clear that $(x_n)_n$ is dense in $X$. To prove that this trajectory witnesses that $f$ is first return recoverable, let $s \in X$ be arbitrary. Using the notation of the Introduction, we need to show that $f(s_n) \to f(s)$.
  It is clear, using the ordering of the trajectory and the fact that the range of $d_X$ is $\{2^{1 - n} : n \ge 0\} \cup \{0\}$, that $\{s_n : n \in \N\} = \{x_{\overline{B}(s, 2^{1 - n}), n} : n \in \N\}$. 
  
  Let $n_k$ be the sequence that lists (in a strictly increasing order) the set $\{n : s_n$ was originally picked$\}$. It follows from Remark \ref{r:infinitely many times Case (a) or (b)} that for infinitely many $n$ the point $s_n$ was originally picked, hence $n_k$ is an infinite sequence. Moreover, since the strategy for Player II is winning, $f(s_{n_k}) \to f(s)$. Let $\varepsilon > 0$ be fixed, and let $N$ be large enough so that $d_Y(f(s_{n_k}), f(s)) < \varepsilon$ if $n_k \ge N$, and we also suppose that $N = n_k$ for some $k$. Now let $n > N$ be arbitrary towards showing that $d_Y(f(s_{n}), f(s)) < \frac{1}{n} + 2\varepsilon$. If $s_{n}$ was originally picked then we are done, so let us suppose otherwise and choose $k$ so that $N \le n_k < n < n_{k + 1}$, and choose $B$, $\ell$ so that $x_{B, \ell} = s_n$. Then $d_Y(f(s_{n}), f(s)) \le d_Y(f(s_n), f(s_{n_k})) + d_Y(f(s_{n_k}), f(s)) \le \frac{1}{n} + d_Y(f(s_{n_k}), f(s)) + d_Y(f(s_{n_k}), f(s)) < \frac{1}{n} + 2\varepsilon$, using the fact that $s, s_n \in B$ and our strategy to pick the point $x_{B, \ell}$ from $B$. Therefore $f(s_n) \to f(s)$, and the proof is complete. 
\end{proof}

\subsection*{Acknowledgements}
The author wishes to thank Hugo Nobrega for introducing him to the eraser game and S\l{}awomir Solecki for suggesting to investigate the connection of the game and first return recoverability. The author is greatly indebted to M\'arton Elekes for numerous helpful discussions and would like to thank Lilla T\'othm\'er\'esz for useful remarks about the manuscript.

\bigskip

\end{document}